\documentclass[11pt]{amsart}
\usepackage{hyperref}
\usepackage{graphicx}
\usepackage{amsmath}
\usepackage{amsfonts}
\usepackage{amssymb}
\usepackage{MnSymbol}
\usepackage{rotating}
\usepackage{amssymb}
\usepackage{gastex}
\usepackage[usenames]{color}
\usepackage[usenames,dvipsnames]{pstricks}
\usepackage{epsfig}
\usepackage{pst-grad} 
\usepackage{pst-plot} 
\newtheorem{thm}{Theorem}[section]
\newtheorem{cor}[thm]{Corollary}
\newtheorem{lem}[thm]{Lemma}
\newtheorem{prop}[thm]{Proposition}

\newcommand{\abs}[1]{\left\vert#1\right\vert}

\newcommand{\C}{minimal non-nilpotent}

\newcommand{\N}{\eta^{\ast}}

\begin{document}
\title[ $\N$-congruence]{The congruence $\N$ on semigroups}
\author{M.H. Shahzamanian}
\address{M.H. Shahzamanian\\ Centro de Matematica Universidade do Porto, Rua do Campo Alegre 687, 4169-007 Porto, Portugal}
\email{ m.h.shahzamanian@fc.up.pt}
\thanks{ 2010
Mathematics Subject Classification. Primary 20F19, 20M07,
Secondary: 20F18.
Keywords and phrases: nilpotent semigroup, nilpotent group, principal series, finite semigroup, pseudovariety.}

\begin{abstract}
In this paper we define a congruence $\N$ on semigroups. For the finite semigroups $S$,  $\N$ is the smallest congruence relation such that $S/\N$ is a nilpotent semigroup (in the sense of Mal'cev).  In order to study the congruence relation $\N$ on finite semigroups, we define a $\textbf{CS}$-diagonal finite regular Rees matrix semigroup.
We prove that, if $S$ is a $\textbf{CS}$-diagonal finite regular Rees matrix semigroup then $S/\N$ is inverse. Also, if $S$ is a completely regular finite semigroup, then $S/\N$ is a Clifford semigroup.

We show that, for every non-null principal factor $A/B$ of $S$, there is a special principal factor $C/D$ such that every element of $A\setminus B$ is $\N$-equivalent with some element of $C\setminus D$.
We call the principal factor $C/D$, the $\N$-root of $A/B$. All $\N$-roots are $\textbf{CS}$-diagonal.  If certain elements of $S$ act in the special way on the $\textbf{R}$-classes of a $\textbf{CS}$-diagonal principal factor then it is not an $\N$-root. Some of these results are also expressed in terms of pseudovarieties of semigroups.

\end{abstract}
\maketitle
\tableofcontents

\section{Introduction}\label{pre}

Freni in \cite{Fre}, for a semigroup $S$,  defined a congruence relation $\gamma^{\ast}$ as follows:\\
Let $\gamma_1 = \{(x, x)\mid x \in S\}$ and, for every integer $n > 1$, $x$ is in the relation $\gamma_n$ with $y$ if and only if there exist elements $z_1, \ldots, z_n$ in $S$ and a permutation $\sigma$ in symmetric group of order $n$ such that $x=\prod_{i=1}^{n}z_i$ and $y=\prod_{i=1}^{n}z_{\sigma(i)}$.
Let $\gamma^{\ast}$ be the transitive closure of $\bigcup_{n \geq 1}\gamma_n$.
He proved that the congruence $\gamma^{\ast}$ is the smallest congruence relation of $S$ such that the quotient $S/\gamma^{\ast}$ is a commutative semigroup.

Mal'cev \cite{Mal} and independently Neumann and Taylor \cite{Neu-Tay} have shown that nilpotent groups can be defined by using semigroup identities (that is, without using inverses). This leads to the notion of a nilpotent semigroup (in the sense of Mal'cev). 

For a semigroup $S$ with elements $x,y,z_{1},z_{2}, \ldots$ one recursively defines two sequences 
$$\lambda_n=\lambda_{n}(x,y,z_{1},\ldots, z_{n})\quad{\rm and} \quad \rho_n=\rho_{n}(x,y,z_{1},\ldots, z_{n})$$ by
$$\lambda_{0}=x, \quad \rho_{0}=y$$ and
$$\lambda_{n+1}=\lambda_{n} z_{n+1} \rho_{n}, \quad \rho_{n+1}=\rho_{n} z_{n+1} \lambda_{n}.$$ A semigroup is said to be \textit{nilpotent} (in the sense of Mal'cev \cite{Mal}) if there exists a positive integer $n$ such that 
$$\lambda_{n}(a,b,c_{1},\ldots, c_{n}) = \rho_{n}(a,b,c_{1},\ldots, c_{n})$$ for all $a,b$ in $S$ and $ c_{1}, \ldots , c_{n}$
in $S^{1}$. The smallest such $n$ is called the nilpotency class
of $S$. Clearly, null semigroups are nilpotent in the sense of Mal'cev. Recall
(\cite{Neu-Tay}) that a semigroup $S$ is said to be
\textit{Neumann-Taylor} (NT)  if, for some $n\geq 2$,
$$\lambda_n(a,b,1,c_2,\ldots,c_n)=\rho_n(a,b,1,c_2,\ldots,c_n)$$
for all $a,b \in S$ and $c_2,\ldots,c_n$ in $S^1$. A semigroup $S$ 
is said to be \textit{positively Engel} (PE) if, for  some $n\geq
2$,
$$\lambda_{n}(a,b,1,1,c,c^{2},\ldots ,c^{n-2})
=\rho_{n} (a,b,1,1,c,c^{2},\ldots ,c^{n-2})$$ for all $a,b$ in $S$
and $c\in S^{1}$. Recall that a pseudovariety of semigroups is a class of finite semigroups closed under taking subsemigroups, homomorphic images and finite direct products. It is easy to verify that the finite nilpotent semigroups, finite Neumann-Taylor semigroups and finite positively Engel semigroups, separately constitute pseudovarieties. We denoted them respectively by $\bf{MN}$, $\bf{NT}$ and $\bf{PE}$.


In this paper, a congruence $\N$ is defined that for a finite semigroup $S$, $\N$ is the smallest congruence relation of $S$ such that $S/\N$ is a nilpotent semigroup.  In order to study the congruence relation $\N$ on finite semigroups, we define a $\textbf{CS}$-diagonal finite regular Rees matrix semigroup and the quotient group $G_{\N}$ of $G/\N$ for the finite regular Rees matrix semigroup where $G$ is the maximal subgroup of it.
We prove that, if $S$ is a $\textbf{CS}$-diagonal finite regular Rees matrix semigroup with the maximal subgroup $G$ then $S/\N$ is an inverse Rees matrix semigroup with maximal subgroup $G_{\N}$. Also if $S$ is a finite completely simple semigroup, then $S/\N \cong G_{\N}$.
Moreover, if $S$ is a completely regular finite semigroup, then $S/\N$ is a Clifford semigroup.

In the section~\ref{approach}, we investigate the congruence relation $\N$ on finite semigroups through their principal series. We define the $\N$-root, for every non-null principal factor of a finite semigroup $S$. Let $S_p / S_{p+1}$ be a principal factor of $S$. If there exists an element in $S_p \setminus S_{p+1}$ such that there is an $\N$-relation between it and $\theta$, the $\N$-root of $S_p/ S_{p+1}$ is $\theta$.
Otherwise its $\N$-root is the $\textbf{CS}$-diagonal principal factor $S_{p'} / S_{p'+1}$ of $S$ such that there is an $\N$-relation between at least two elements of them and there is no $\N$-relation between any elements of the $\N$-root and $S_{p'+1}$. We prove that 
if $\theta$ is the $\N$-root of $S_p / S_{p+1}$, then the subset $S_p \setminus S_{p+1}$ is in the class of $\theta$ of $\N$.
All $\N$-roots are $\textbf{CS}$-diagonal. If certain elements of $S$ act in the special way on the $\textbf{R}$-classes of a $\textbf{CS}$-diagonal principal factor $S_p / S_{p+1}$, then $S_p / S_{p+1}$ is not an $\N$-root. 
Although it remains as open problem the conditions of acting of elements of $S \setminus S_{p+1}$ on the $\textbf{R}$-classes of $S_p / S_{p+1}$ when it is not its $\N$-root. In the special case, when $S$ is semisimple, the classes of $S/\N$ are $\{\theta\}$ (if $S$ has $\theta$) and the $\N$-classes of $\N$-root principal factors of $S$.

Finally we compare the pseudovariety $\textbf{MN}$ with the pseudovarieties $\textbf{BG}_{nil}$ and $\textbf{BI}$ and got that
$ \textbf{MN} \subset \textbf{BG}_{nil}$ but $\textbf{MN} \not\subset \textbf{BI}$ and $\textbf{BI} \not\subset \textbf{MN}$.


For standard notations and terminology of finite semigroups, refers to \cite{Cli}.
A completely $0$-simple finite semigroup $S$ is isomorphic with a
regular Rees matrix semigroup $\mathcal{M}^{0}(G, n,m;P)$, where
$G$ is a maximal subgroup of $S$, $P$ is the $m\times n$ sandwich
matrix with entries in $G^{\theta}$ and $n$ and $m$ are positive
integers. The nonzero elements of $S$ is denoted by $(g;i,j)$, where $g\in
G$, $1\leq i \leq n$ and $1\leq j\leq m$; the zero
element is simply denoted $\theta$. The element  of $P$ on the
$(j,i)$-position is denoted by $p_{ji}$. The set of non zero elements is
denoted by $\mathcal{M} (G,n,m;P)$. If all elements of $P$ are
non zero then this is a semigroup and every completely simple
finite semigroup is of this form. If $P=I_{n}$, the identity
matrix, then $S$ is an inverse semigroup. Jespers and Okninski prove that a completely $0$-simple semigroup
$\mathcal{M}^{0}(G,n,m;P)$ is nilpotent if and only if $n=m$,
$P=I_{n}$ and $G$ is a nilpotent group [\cite{Jes-Okn}, Lemma 2.1].

Let $S=\mathcal{M}^0(G,A,B;P)$ be a Rees matrix semigroup. We assume $A$ and $B$ are disjoint. The incidence graph of $S$, denoted $\Gamma(S)$, has vertex set $V = A \cup B$. The edge set is given by $$E = \{(a, b), (b, a) \in (A \times B) \cup (B \times A) \mid p_{b,a}\neq \theta\}.$$

\section{The relation $\N$ on finite semigroups}

The following lemma gives a necessary and sufficient condition for a finite semigroup not to be nilpotent \cite{Jes-Sha}.

\begin{lem} \label{finite-nilpotent}
A finite semigroup $S$ is not nilpotent if and only if there exists a positive integer $m$, distinct elements $x, y\in S$ and elements $w_{1}, w_{2}, \ldots, w_{m}\in S^{1}$ such that $x = \lambda_{m}(x, y, w_{1}, w_{2}, \ldots, w_{m})$, $y = \rho_{m}(x,y, w_{1}, w_{2}, \ldots, w_{m})$.
\end{lem}

This lemma motivates the definition of a congruence relation $\N$ on finite semigroups which is refinement of $\gamma^{\ast}$.

Let $S$ be a semigroup.
Let $\eta_1 = \{(x, x)\mid x \in S\}$ and, for every integer $n > 1$, $x$ has a relation $\eta_n$ with $y$ if and only if there exist elements $z_1, \ldots, z_n$ in $S^1$  such that $x=\lambda_{n}(x,y,z_{1},\ldots, z_{n})$ and $y=\rho_{n}(x,y,z_{1},\ldots, z_{n})$.
Let $\eta$ be the transitive closure of $\bigcup_{i \geq 1}\eta_i$ and  $\N$ the smallest congruence relation including $\eta$.

\begin{lem} \label{smallest-congruence}
Let $S$ be a finite semigroup. The relation $\N$ is the smallest congruence relation of $S$ such that $S/\N$ is a nilpotent semigroup.
\end{lem}

\begin{proof}
First we prove that $S/\N$ is nilpotent. 

Suppose the contrary. Since $S/\N$ is finite, by Lemma~\ref{finite-nilpotent}, there exist distinct classes $\N(x), \N(y)\in S/\N$ and elements $\N(w_{1}), \N(w_{2}), \ldots, \N(w_{m})\in (S/\N)^{1}$ such that $$\N(x) = \lambda_{m}(\N(x), \N(y), \N(w_{1}), \N(w_{2}), \ldots, \N(w_{m})),$$ $$\N(y) = \rho_{m}(\N(x),\N(y), \N(w_{1}), \N(w_{2}), \ldots, \N(w_{m})).$$
Then for every elements $x'\in\N(x), y'\in\N(y)$, there exists element $x''\in\N(x), y''\in\N(y)$, such that $$x'' = \lambda_{m}(x', y', w_{1}, w_{2}, \ldots, w_{m}) \mbox{~and~} y'' = \rho_{m}(x',y', w_{1}, w_{2}, \ldots, w_{m}).$$
Since $\abs{S\times S}$ is finite, there exist elements $\alpha\in\N(x), \beta\in\N(y)$, an integer $k=k'm$ and elements $z_1,\ldots,z_k\in S^1$ such that $$\alpha = \lambda_{k}(\alpha, \beta, z_{1}, z_{2}, \ldots, z_{k}) \mbox{~and~} \beta = \rho_{k}(\alpha,\beta, z_{1}, z_{2}, \ldots, z_{k}).$$ Then $\N(\alpha)=\N(\beta)$, a contradiction.

Now suppose that $\rho$ is a congruence relation on $S$ such that $R = S/\rho$ is nilpotent and $x \eta_n y$ for some $x,y \in S$ and  $n \in \mathbb{N}$. Thus there exist elements $z_1, \ldots, z_n$ in $S^1$  such that $$x=\lambda_{n}(x,y,z_{1},\ldots, z_{n}) \mbox{~and~} y=\rho_{n}(x,y,z_{1},\ldots, z_{n}).$$ If $\rho(x) \neq \rho(y)$, then Lemma~\ref{finite-nilpotent} yields $$\langle \rho(x), \rho(y), \rho(z_1), \ldots, \rho(z_n) \rangle$$ is a non-nilpotent subsemigroup  of $R$ and thus $R$ is not nilpotent, a contradiction. Therefore  $x \rho y$ and thus $\N \subseteq \rho$. 
\end{proof}

Neumann and Taylor prove that a group $G$ is nilpotent if and only if it is nilpotent in the sense of Mal'cev [\cite{Neu-Tay}, Corollary 1].

{\bf Remark.} Aghabozorgi, Davvaz and Jafarpour in~\cite{Agh-Dav-Jaf} introduce the smallest equivalence relation $\nu^{\ast}$ on a hypergroup $H$.
We mention this definition limited to the case of semigroups.

Let $L_0(H) = H$ and 
$$L_{k+1}(H) = \{h \in H \mid xy=hyx \mbox{ such that } x \in L_k(H) \mbox{ and } y \in H\},$$
for all $k \geq 0$. Suppose that $n \in \mathbb{N}$ and $\nu_n = \bigcup_{m>1} \nu_{m,n}$, where $\nu_{1,n}$ is the diagonal relation and for
every integer $m \geq 1$, $\nu_{m,n}$ is the relation defined as follows:
$$x \nu_{m,n} y \Leftrightarrow \exists (z_1, \ldots, z_m) \in H^m, \exists \sigma \in \mathbb{S}_m : \sigma(i) = i \mbox{ if } z_i  \notin L_n(H)$$ 
$$\mbox{ such that }x = \prod^m_{i=1}z_i \mbox{ and } y = \prod^m_{i=1}z_{\sigma(i)}.$$
Let $\nu^{\ast}_n$ be the transitive
closure of $\nu_n$ and  the relation $\nu^{\ast}$ as follows:
$$\nu^{\ast} = \bigcap_{n \geq 1} \nu^{\ast}_n.$$

They prove that the relation $\nu^{\ast}$ is the smallest congruence relation on a finite group $H$ such that $H/\nu^{\ast}$
is a nilpotent group. Thus Lemma~\ref{smallest-congruence} and Corollary 1 of \cite{Neu-Tay}, follow that $\nu^{\ast}=\N$ for the case of finite groups. 

In general, the relations $\nu^{\ast}$ and $\N$ are not equal. For example, let $M$ be a regular Rees matrix semigroup $M=\mathcal{M}^{0}(\{1\},2,2;I_2)$. Based of Lemma 2.1 of~\cite{Jes-Okn}, $M$ is nilpotent.
So $M / \N=M$ (Lemma~\ref{smallest-congruence}). 

Since $(1;1,1),(1;2,2)\in L_i(M)$ for all $1 \leq i$,
$$(1;1,1)= (1;1,2)(1;2,2)(1;2,1)(1;1,1) \mbox{~and}$$
$$\theta= (1;1,2)(1;1,1)(1;2,1)(1;2,2),$$
we have $(1;1,1)\nu^{\ast} \theta$. Then $M / \nu^{\ast} \neq M$ and thus $M / \N \neq M / \nu^{\ast}.$ 

Let $M= \mathcal{M}^0(G,n,m;P)$
be a finite regular Rees matrix semigroup with the sandwich matrix $P$. We call $M$ is $\textbf{CS}$-diagonal, if
$p_{r,t}$, $p_{r',t}$ and $p_{r,t'}$ are non zero then $p_{r',t'}$ is non zero
for all $1 \leq r,r'\leq m$ and $1 \leq t,t'\leq n$. 

The above assumption let us to define the following equivalence relation on the sets $\{1,\ldots,n\}$. The integers $1 \leq i_{\alpha}, i_{\beta}\leq n$ are in the same class if and only if there exists an integer $1 \leq j \leq m$ such that $p_{j,i_{\alpha}}$ and $p_{j,i_{\beta}}$ are non zero. Suppose that the set $\{1,\ldots,n\}$ partition to $t$ classes $\mathbb{I}_1, \ldots, \mathbb{I}_t$. We define the set $\mathbb{I'}_l$ for $1 \leq l \leq t$. If $p_{j,i} \neq \theta$ for $i \in \mathbb{I}_l$, then $j \in \mathbb{I'}_l$.
It is easy to verify that the set $\{1,\ldots,m\}$ partitions to the $t$ classes $\mathbb{I'}_l$ for $1 \leq l \leq t$. Also
it is easy to verify that if $i \in \mathbb{I}_l$ and $j \in \mathbb{I'}_l$ for $1 \leq l \leq t$ then $p_{j,i} \neq \theta$ and 
if $i \in \mathbb{I}_g$, $j \in \mathbb{I'}_h$ and $g \neq h$ then $p_{j,i} = \theta$, for every $1 \leq g,h \leq t$. We show the number of classes of $\{1,\ldots,n\}$, by $n_{\N}$.
\psscalebox{1.0 1.0} 
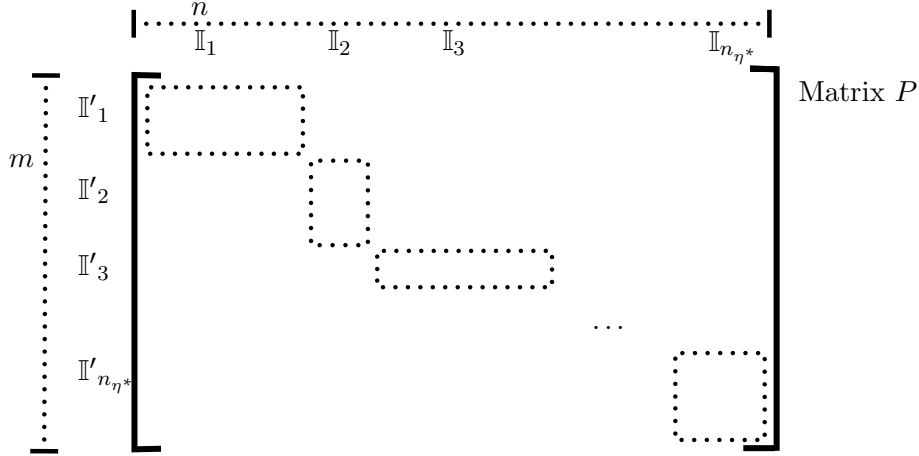
{
\begin{pspicture}(0,-2.9906626)(12.94447,2.9906626)
\rput[bl](10.496289,1.6804){Matrix $P$}
\psline[linecolor=black, linewidth=0.08](2.0235295,2.0228193)(1.6705883,2.0306625)(1.6705883,-2.9497297)(2.0235295,-2.941573)
\psline[linecolor=black, linewidth=0.08](10.207541,2.1072586)(9.8546,2.1152778)(10.207541,2.1072586)(10.207541,-2.9689574)
\psframe[linecolor=black, linewidth=0.06, linestyle=dotted, dotsep=0.10583334cm, dimen=outer](3.9397647,1.8970155)(1.8117647,0.95301545)
\psframe[linecolor=black, linewidth=0.06, linestyle=dotted, dotsep=0.10583334cm, dimen=outer](4.803765,0.92101544)(3.9877646,-0.26298457)
\psframe[linecolor=black, linewidth=0.06, linestyle=dotted, dotsep=0.10583334cm, dimen=outer](7.251765,-0.27898455)(4.8677645,-0.8229846)
\psframe[linecolor=black, linewidth=0.06, linestyle=dotted, dotsep=0.10583334cm, dimen=outer](10.079059,-1.6389846)(8.8310585,-2.8549845)
\rput[bl](7.756235,-1.3509846){$\ldots$}
\rput[bl](2.4696472,2.3008978){$\mathbb{I}_1$}
\rput[bl](4.245647,2.3008978){$\mathbb{I}_2$}
\rput[bl](5.763765,2.3008978){$\mathbb{I}_3$}
\rput[bl](9.304471,2.1784272){$\mathbb{I}_{n_{\N}}$}
\rput[bl](0.9204706,1.4066625){$\mathbb{I'}_1$}
\rput[bl](0.9204706,0.32525074){$\mathbb{I'}_2$}
\rput[bl](0.9204706,-0.6959257){$\mathbb{I'}_3$}
\rput[bl](0.9204706,-2.206514){$\mathbb{I'}_{n_{\N}}$}
\psline[linecolor=black, linewidth=0.06, linestyle=dotted, dotsep=0.10583334cm, tbarsize=0.07055555555555555cm 5.0]{|-|}(0.49411765,2.0541918)(0.47058824,-3.0046315)
\psline[linecolor=black, linewidth=0.06, linestyle=dotted, dotsep=0.10583334cm, tbarsize=0.07055555555555555cm 5.0]{|-|}(1.6310588,2.7130156)(10.141176,2.7130156)
\rput[bl](2.4235294,2.8306625){$n$}
\rput[bl](0.0,0.8071331){$m$}
\psline[linecolor=black, linewidth=0.08](10.188235,-2.9340434)(9.764706,-2.9340434)
\end{pspicture}
}


Suppose that $S$ is a semigroup. 
The subsemigroup $\langle E(S)\rangle$ denotes the subsemigroup generated by $E(S)$, the set of idempotents of $S$. For a pseudovariety $\textbf{V}$, recall from \cite{Alm} the pseudovariety $$\textbf{EV} = \{S \in \textbf{S}\mid\langle E(S)\rangle \in \textbf{V}\}.$$
The pseudovariety $\textbf{DS}$ can be specified, $S \not\in \textbf{DS}$ if and only if there exist idempotents $e, f \in S$ in the same $\textbf{J}$-class such that $ef$ and $fe$ are not both in that $\textbf{J}$-class (\cite[Exercise 8.1.6]{Alm}).
Then $$\textbf{EDS}= \{S \in \textbf{S}\mid\langle E(S)\rangle \in \textbf{DS}\}.$$
It is well known that (for example \cite[Exercise 4.13.38]{Rho-Ste}), a semigroup $S$ belongs to $\textbf{EDS}$ if and only if, for each regular $\textbf{J}$-class $J$, the connected components of the incidence graph $\Gamma(J^0)$ are complete bipartite graphs. Suppose that the finite regular Rees matrix semigroup $M= \mathcal{M}^0(G,n,m;P)$ is $\textbf{CS}$-diagonal. Since the connected component $\mathbb{I}_i\cup \mathbb{I'}_i$ is a complete bipartite graph for every $1\leq i \leq n_{\N}$, $M\in \textbf{EDS}$.

Let $G_{\N}$ be the smallest quotient group of $G/\N$ such that for every integers $1 \leq a,b \leq n_{\N}$ $$p_{x,c'}^{-1}p_{x,c}=p_{\beta,c'}^{-1}p_{\beta,c} \mbox{ and } p_{d,y}p_{d',y}^{-1}=p_{d,\alpha}p_{d',\alpha}^{-1}$$ for every integers $c,c' \in \mathbb{I}_a$, $\beta,x \in \mathbb{I'}_a$, $\alpha, y \in \mathbb{I}_b$ and $d, d' \in \mathbb{I'}_b$. We denote its quotient map by $\phi_{\N}:G/\N \rightarrow G_{\N}$, 
the equivalence relation class of $p_{\beta,c'}^{-1}p_{\beta,c}$ by $\beta_{c',c}$ and the equivalence relation class of $p_{d,\alpha}p_{d',\alpha}^{-1}$ by $\alpha_{d',d}$ in $G_{\N}$.

It is easy to verify that for every $1 \leq a,b \leq n_{\N}$ $$\beta_{c',c'}=\beta_{c',c}\beta_{c,c'}=1,\alpha_{d,d}=\alpha_{d',d}\alpha_{d,d'}=1,$$
$$\beta_{c'',c'}\beta_{c',c}=\beta_{c'',c} \mbox{ and }\alpha_{d',d}\alpha_{d'',d'}=\alpha_{d'',d}$$
for every integers $c,c',c'' \in \mathbb{I}_a$, and $d, d',d'' \in \mathbb{I'}_b$.

\begin{prop} \label{regular Rees matrix semigroup}
Let $M= \mathcal{M}^0(G,n,m;P)$
be a finite regular Rees matrix semigroup with the sandwich matrix $P$. If $M$ is not $\textbf{CS}$-diagonal, then $M/\N= \{\theta\}$.
Otherwise $M/\N \cong \mathcal{M}^0(G_{\N},n_{\N},n_{\N};I_{n_{\N}})$.
\end{prop}

\begin{proof}
First suppose that $M$ is nilpotent. Then $n=m$ and $$M \cong\mathcal{M}^0(G,n,n;I_n),$$ where $I_n$ denotes the
identity $n$-by-$n$ matrix [\cite{Jes-Okn}, Lemma 2.1]. Since $M$ is nilpotent then Lemma~\ref{smallest-congruence} follows that  $M=M/\N$ and thus  $M/\N \cong \mathcal{M}^0(G,n,n;I_n)$.

Suppose that $M$ is not nilpotent. If $p_{j,i}$ and $p_{j',i}$ are non zero and $j \neq j'$ for some $1 \leq i \leq n$ and $1 \leq j,j' \leq m$, then
$$(p_{j,i}^{-1};i,j) = \lambda_2((p_{j,i}^{-1};i,j),(p_{j',i}^{-1};i,j'),1,1),$$
$$(p_{j',i}^{-1};i,j') = \rho_2((p_{j,i}^{-1};i,j),(p_{j',i}^{-1};i,j'),1,1)$$
and thus $(p_{j,i}^{-1};i,j) \eta_2 (p_{j',i}^{-1};i,j')$. Let $1 \leq i' \leq n$ such that $p_{j',i'}=\theta$ and $p_{j,i'}\neq \theta$. Since the relation $\N$ is congruence, by $(p_{j,i}^{-1};i,j) \eta_2 (p_{j',i}^{-1};i,j')$, 
$$(p_{j,i}^{-1};i,j)(p_{j,i'}^{-1}p_{j,i};i',j) \eta_2 (p_{j',i}^{-1};i,j')(p_{j,i'}^{-1}p_{j,i};i',j)$$ and thus $(1;i,j) \N \theta$. Hence
for every $1 \leq a \leq n$, $1 \leq b \leq m$ and $g \in G$, we have $$(g;a,b)=(gp_{j',i}^{-1};a,j')(1;i,j)(p_{j,i}^{-1};i,b) \N \theta.$$
Therefore $M/\N= \{\theta\}$. 


Now suppose that $M$ is $\textbf{CS}$-diagonal. We define the relation $\kappa$ on $M$ as follows: 
$$\theta \kappa \theta$$
and for every integers $1 \leq a,b \leq n_{\N}$,  integers
$c, c' \in \mathbb{I}_a$,
$d, d' \in \mathbb{I'}_b$
and $g \in G$, $$(g;c,d) \kappa (\beta_{c',c}g'\alpha_{d',d};c',d') \mbox{ if and only if } \phi_{\N}(g)=\phi_{\N}(g').$$

We claim that the relation $\kappa$ is congruence. 

Suppose that $(g;c,d)$ is a non zero element of $M$. Then there exist integers
$1 \leq a,b \leq n_{\N}$ such that $c \in \mathbb{I}_a$ and $d \in \mathbb{I'}_b$. Since $\beta_{c,c}=\alpha_{d,d}=1$, $(g;c,d) \kappa (\beta_{c,c}g\alpha_{d,d};c,d)$.
Then $(g;c,d) \kappa (g;c,d)$ and thus $\kappa$ is reflexive.

Now suppose that $(g;c,d) \kappa (g';c',d')$. Then there exist integers $1 \leq a,b \leq n_{\N}$ such that $c, c' \in \mathbb{I}_a$,
$d, d' \in \mathbb{I'}_b$ and
$$\phi_{\N}(g')= \beta_{c',c}\phi_{\N}(g)\alpha_{d',d}.$$
Then $$\phi_{\N}(g)= \beta_{c',c}^{-1}\phi_{\N}(g')\alpha_{d',d}^{-1}=\beta_{c,c'}\phi_{\N}(g')\alpha_{d,d'}$$
 and thus $(g';c',d') \kappa (g;c,d)$. Therefore the relation $\kappa$ is symmetric.

If $(g;c,d) \kappa (g';c',d')$ and $(g';c',d') \kappa (g'';c'',d'')$ then 
there exist integers $1 \leq a,b \leq n_{\N}$, 
such that $c, c',c'' \in \mathbb{I}_a$,
$d, d',d'' \in \mathbb{I'}_b$,
$$\phi_{\N}(g')= \beta_{c',c}\phi_{\N}(g)\alpha_{d',d}$$
and
$$\phi_{\N}(g'')= \beta_{c'',c'}\phi_{\N}(g')\alpha_{d'',d'}.$$
Then
$$\phi_{\N}(g'')= \beta_{c'',c'}\beta_{c',c}\phi_{\N}(g)\alpha_{d',d}\alpha_{d'',d'}=\beta_{c'',c}\phi_{\N}(g)\alpha_{d'',d}$$
and thus
 $(g;c,d) \kappa (g'';c'',d'')$. Therefore the relation $\kappa$ is transitive.

Then the relation $\kappa$ is an equivalence relation. Now we investigate that $\kappa$ is congruence. Suppose that $(g;c,d) \kappa (g';c',d')$ and $(g'';c'',d'') \in M$.
Then there exist integers $1 \leq a,b \leq n_{\N}$ such that $c, c' \in \mathbb{I}_a$,
$d, d' \in \mathbb{I'}_b$ and
$$\phi_{\N}(g')= \beta_{c',c}\phi_{\N}(g)\alpha_{d',d}.$$
If $d'' \notin \mathbb{I'}_a$ then 
$(g'';c'',d'')(g;c,d)=(g'';c'',d'')(g';c',d')=\theta$. Otherwise $d'' \in \mathbb{I'}_a$. 
Since
$\phi_{\N}(g')= \beta_{c',c}\phi_{\N}(g)\alpha_{d',d}$, $$\phi_{\N}(g')= \phi_{\N}(p_{d'',c'}^{-1})\phi_{\N}(p_{d'',c})\phi_{\N}(g)\alpha_{d',d}$$ and thus
$$\phi_{\N}(g''p_{d'',c'}g')= \phi_{\N}(g''p_{d'',c}g)\alpha_{d',d}.$$
Then
$$(g''p_{d'',c}g;c'',d) \kappa (g''p_{d'',c'}g';c'',d').$$
Hence
$$(g'';c'',d'')(g;c,d) \kappa (g'';c'',d'')(g';c',d').$$ 
Similarly
$$(g;c,d)(g'';c'',d'') \kappa (g';c',d')(g'';c'',d'').$$
Therefore the relation $\kappa$ is congruence.

Let $i_{\N}$ and $i'_{\N}$ are the smallest integers of $\mathbb{I}_i$ are $\mathbb{I'}_i$, respectively for every $1\leq i\leq n_{\N}$.
We claim that $M/\kappa \cong \mathcal{M}^0(G_{\N},n_{\N},n_{\N};P')$ with $P'=[p'_{i,j}]$ is a $n_{\N} \times n_{\N}$ matrix such that if $i\neq j$, then $p'_{i,j}=\theta$  otherwise $p'_{i,j}=p_{i'_{\N},i_{\N}}$ for every $1\leq i,j\leq n_{\N}$. We define the homomorphism $\phi: M/\kappa\rightarrow \mathcal{M}^0(G_{\N},n_{\N},n_{\N};P')$, as
$$\phi(\theta)=\theta,$$
$$\phi((g;c,d))=(\beta_{a_{\N},c}
\phi_{\N}(g)\alpha_{b'_{\N},d};a_{\N},
b'_{\N})$$
such that $c \in \mathbb{I}_a$, $d \in \mathbb{I'}_b$  
for every $(g;c,d) \in M/\kappa$.

First we prove that $\phi$ is well-defined.  Suppose that $(g;c,d) \kappa (g';c',d')$. Then there exist integers $1 \leq a,b \leq n_{\N}$ such that $c, c' \in \mathbb{I}_a$,
$d, d' \in \mathbb{I'}_b$ and
$$\phi_{\N}(g')= \beta_{c',c}\phi_{\N}(g)\alpha_{d',d}.$$
Then
$$\beta_{a_{\N},c'}
\phi_{\N}(g')
\alpha_{b'_{\N},d'}=
 \beta_{a_{\N},c'}\beta_{c',c}\phi_{\N}(g)\alpha_{d',d}\alpha_{b'_{\N},d'},$$ and thus
$$\beta_{a_{\N},c'}
\phi_{\N}(g')\alpha_{b'_{\N},d'}=
 \beta_{a_{\N},c}
\phi_{\N}(g)\alpha_{b'_{\N},d}.$$ Therefore  $\phi((g;c,d)) = \phi((g';c',d')).$
Then $\phi$ is well-defined.


Now suppose that $g,g' \in G$, $c \in \mathbb{I}_a$, $d \in \mathbb{I'}_e$, $c' \in \mathbb{I}_e$ and $d' \in \mathbb{I'}_b$. Hence  
$$\phi((g;c,d)(g';c',d'))=\phi((gp_{d,c'}g';c,d'))=$$
$$(\beta_{a_{\N},c}
\phi_{\N}(gp_{d,c'}g')\alpha_{b'_{\N},d'};a_{\N},
b'_{\N}).$$
Since $$\phi_{\N}(g)\alpha_{e'_{\N},d}
\phi_{\N}(p_{e'_{\N},e_{\N}})
\beta_{e_{\N},c'}\phi_{\N}(g')=$$ $$
\phi_{\N}(gp_{d,e_{\N}}p_{e'_{\N},
e_{\N}}^{-1}p_{e'_{\N},
e_{\N}}p_{d,e_{\N}}^{-1}p_{d,c'}g'),$$
the last statement is equal to
$$(\beta_{a_{\N},c}
\phi_{\N}(g)\alpha_{e'_{\N},d}
\phi_{\N}(p_{e'_{\N},e_{\N}})
\beta_{e_{\N},c'}\phi_{\N}(g')\alpha_{b'_{\N},d'};a_{\N},
b'_{\N})=$$
$$(\beta_{a_{\N},c}
\phi_{\N}(g)\alpha_{e'_{\N},d};a_{\N},
e'_{\N})(\beta_{e_{\N},c'}
\phi_{\N}(g')\alpha_{b'_{\N},d'};e_{\N},
b'_{\N})=$$
$$\phi((g;c,d))\phi((g';c',d')).$$
Therefore $\phi$ is a homomorphism.


If $\phi((g;c,d))=\phi((g';c',d'))$, then there exist $1 \leq a,b \leq n_{\N}$ such that
$c,c' \in \mathbb{I}_a$, $d,d' \in \mathbb{I'}_b$
and
$$(\beta_{a_{\N},c}
\phi_{\N}(g)\alpha_{b'_{\N},d};a_{\N},
b'_{\N})=(\beta_{a_{\N},c'}
\phi_{\N}(g')\alpha_{b'_{\N},d'};a_{\N},
b'_{\N}).$$
Hence $$\beta_{a_{\N},c}
\phi_{\N}(g)\alpha_{b'_{\N},d}=\beta_{a_{\N},c'}
\phi_{\N}(g')\alpha_{b'_{\N},d'}$$
and thus
$$
\phi_{\N}(g)=
\beta_{a_{\N},c}^{-1}\beta_{a_{\N},c'}
\phi_{\N}(g')\alpha_{b'_{\N},d'}\alpha_{b'_{\N},d}^{-1}
=\beta_{c,c'}\phi_{\N}(g')\alpha_{d,d'}.$$
Then $(g;c,d)\kappa (g';c',d')$ and thus $\phi$ is one to one.


Obviously $\phi$ is onto. Therefore $\phi$ is isomorphism.

Therefore $M/\kappa \cong \mathcal{M}^0(G_{\N},n_{\N},n_{\N};P')$ and it is easy to verify that $M/\kappa \cong \mathcal{M}^0(G_{\N},n_{\N},n_{\N};I_{n_{\N}})$.

Since the group $G/\N$ is nilpotent, the quotient group $G_{\N}$ is nilpotent and by [\cite{Jes-Okn}, Lemma 2.1], $M/\kappa$ is nilpotent and thus Lemma~\ref{smallest-congruence} follows that $\N \subseteq \kappa$.


Suppose that $c,c' \in \mathbb{I}_a$ and $d, d' \in \mathbb{I'}_b$  for integers $1 \leq a,b \leq n_{\N}$. Let integers $e,e' \in \mathbb{I'}_a$ and $f, f' \in \mathbb{I}_b$.
Thus $$(p_{d,f}^{-1};f,d) \eta_2 (p_{d',f}^{-1};f,d') \mbox{ and }(p_{e',c}^{-1};c,e') \eta_2 (p_{e',c'}^{-1};c',e').$$ 

Since the relation $\eta^{\ast}$ is congruence, $$(g;c,d)(p_{d,f}^{-1};f,d) \eta^{\ast} (g;c,d)(p_{d',f}^{-1};f,d') \mbox{ and }$$   $$(p_{e',c}^{-1};c,e')(gp_{d,f}p_{d',f}^{-1};c,d') \eta^{\ast} (p_{e',c'}^{-1};c',e')(gp_{d,f}p_{d',f}^{-1};c,d')$$ and thus
$$
(g;c,d) \eta^{\ast} (gp_{d,f}p_{d',f}^{-1};c,d') \mbox{ and }
(gp_{d,f}p_{d',f}^{-1};c,d') \eta^{\ast} (p_{e',c'}^{-1}p_{e',c}gp_{d,f}p_{d',f}^{-1};c',d')$$
for every $g \in G$. However, as $\N$ is transitive,
$$(g;c,d) \eta^{\ast} (p_{e',c'}^{-1}p_{e',c}gp_{d,f}p_{d',f}^{-1};c',d')$$
for every $g \in G$.


Now suppose that $(h;r,t) \kappa (h';r',t')$. Then there exist integers $1 \leq o,p \leq n_{\N}$ such that $r, r' \in \mathbb{I}_o$,
$t, t' \in \mathbb{I'}_p$ and
$\phi_{\N}(h')= \beta_{r',r}\phi_{\N}(h)\alpha_{t',t}.$
Then there exist integers $w \in \mathbb{I}_p$, $v \in \mathbb{I'}_o$ such that $h' \N p_{v,r'}^{-1}p_{v,r}hp_{t,w}p_{t',w}^{-1}$ and hence
$$(h';r',t') \N(p_{v,t'}^{-1}p_{v,t}hp_{t,w}p_{t',w}^{-1};r',t').$$

Now, by above $ (h;r,t)\N(p_{v,t'}^{-1}p_{v,t}hp_{t,w}p_{t',w}^{-1};r',t')$  and thus $$(h;r,t) \N (h';r',t').$$
Hence $\kappa \subseteq \N$. 

Therefore $M/\N \cong \mathcal{M}^0(G_{\N},n_{\N},n_{\N};I_{n_{\N}})$ and the result follows.
\end{proof}


The following result can be seen as an immediate corollary of Proposition~\ref{regular Rees matrix semigroup}.

\begin{cor} \label{completely simple semigroup}
Let $M= \mathcal{M}(G,n,m;P)$ be a finite completely simple semigroup with the sandwich matrix $P$. Then $M/\N \cong G_{\N}$.
\end{cor}

Now we recall some definitions (see for example \cite{Cli} and \cite{How}).

Suppose $S$ is a semigroup  such that $S= \bigcup \{S_{\alpha
}\mid \alpha \in \Omega \}$, a disjoint union of subsemigroups
$S_{\alpha}$, and  such that for every pair of elements $\alpha,
\beta \in \Omega$ we have $S_{\alpha}S_{\beta} \subseteq
S_{\gamma}$ for some $\gamma \in \Omega$. One then has a product
in $\Omega$ defined by $\alpha \beta = \gamma$ if
$S_{\alpha}S_{\beta} \subseteq S_{\gamma}$ and one says that $S$
is the union of the band $\Omega$ of semigroups $S_{\alpha}$, with
$\alpha \in \Omega$. If $\Omega$ is commutative, then one obtains
a partial order relation $\leq $ on  $\Omega$ with $\beta
\leq\alpha$ if $\alpha \beta =\beta$. In this case  $\Omega$ is a
semilattice and one says that $S$ is the semilattice $\Omega$
of semigroups $S_{\alpha}$.

A semigroup $S$ is a completely regular semigroup in which every element is in some subgroup of the semigroup. Therefore every $H$-class in $S$ is a group.
One important class of completely regular semigroups, namely the class of completely simple semigroups.
Completely 0-simple semigroups are not in general completely regular, since not every $H$-class in such a semigroup is a group.
If $S$ is completely simple then $S$ is completely regular and simple. Now suppose that $S$ is a completely regular semigroup.
The green relation $J$ is a congruence and we denote the semilattice $S /J$ by $Y$. For each $a \in Y$, $S_a$ is a $J$-class of $S$ and is a completely simple subsemigroup.
Thus $S$ is the disjoint union of the completely simple semigroups and the congruence property of $J$ gives us that $S_aS_b \subseteq S_{ab}$.
We say that $S$ is a semilattice of completely simple semigroups. 
A Clifford semigroup is a semigroup that is both an inverse semigroup and a completely regular semigroup.

The following lemma is the preliminary result toward the identification of the $\N$-quotient of completely regular semigroups.

\begin{lem} \label{completely regular}
Let $S$ be a semilattice of completely simple semigroups $M_i=\mathcal{M}(G_i,n_i,m_i;P_i)$, for $1\leq i \leq n$. Then there exist the groups $G'_i$ such that $G'_i$ is a quotient group of the group $G_{i_{\N}}$, for $1\leq i \leq n$ and $S/\N$ is the semilattice of them.
\end{lem}

\begin{proof}
Suppose that $a \N b$. Then there exist elements $a_1,\ldots, a_k, b_1, \ldots, b_k \in S$ such that  $a =a_1\ldots a_k$, $b=b_1\ldots b_k$ and $a_q \eta b_q$, for $1 \leq q \leq k$. As $a_q \eta b_q$, there exist integers $n_q$, $q_1, \ldots, q_{n_q}$ and elements $c_{q,0}, c_{q,1}, \ldots, c_{q,n_q}$ in $S$ such that 
$$a_q=c_{q,0} \eta_{q_{1}} c_{q,1},~ c_{q,1} \eta_{q_{2}} c_{q,2},~ \ldots,~ c_{q,n_{q-1}} \eta_{q_{n_q}} c_{q,n_q}=b_q$$  for $1 \leq q \leq k$.
Since $c_{q,p-1} \eta_{q_{p}} c_{q,p}$, there exist elements $d_{q,p,1}, \ldots, d_{q,p,q_p}$ in $S^1$  such that $$c_{q,p-1}=\lambda_{q_p}(c_{q,p-1},c_{q,p},d_{q,p,1},\ldots, d_{q,p,q_p})\mbox{~and~} $$ $$c_{q,p}=\rho_{q_p}(c_{q,p-1},c_{q,p},d_{q,p,1},\ldots, d_{q,p,q_p}),$$ for every $1 \leq q \leq k$ and $1 \leq p \leq n_q$. 

Suppose that $x \in M_{i}$, $y \in M_{j}$, $1\leq i,j\leq n$ and $x\eta_m y$, for some $x,y \in S$. Then there exists a positive integer $m$ and elements $w_{1}, w_{2}, \ldots, w_{m}\in S^{1}$ such that $x = \lambda_{m}(x, y, w_{1}, w_{2}, \ldots, w_{m})$, $y = \rho_{m}(x,y, w_{1}, w_{2}, \ldots, w_{m})$. 
Since $S$ is semilattice of $M_i$ for $1\leq i\leq n$, there exists $1\leq i_1,\ldots,i_{m'}\leq n$ such that $m'\leq m$ and $i=iji_1\ldots i_{m'}$ and $j=iji_1\ldots i_{m'}$. Then $i=j$. 

Now by above, there exists an integer $1 \leq i_q \leq n$ such that $a_q, b_q \in M_{i_q}$ for every  $1 \leq q \leq k$. Again as $S$ is semilattice of $M_i$ for $1\leq i\leq n$, there exists an integer $1\leq g \leq n$ such that $a,b\in M_g$. 

As every $M_i$ is a completely simple semigroup, for $1\leq i\leq n$, by Corollary~\ref{completely simple semigroup}, $M_i/\N=G_{i_{\N}}$.

Finally as above results, there exist the groups $G'_i$ such that $G'_i$ is a quotient group of the group $G_{i_{\N}}$ for $1\leq i \leq n$ and $S/\N= \bigcup_{1\leq i \leq n}G'_i$.
\end{proof}

In~\cite{Jes-Sha}, Jespers and the author investigated the upper non-nilpotent graph ${\mathcal N}_{S}$ of a finite semigroup $S$. Recall that the vertices of ${\mathcal N}_{S}$ are the elements of $S$ and there is an edge between $x$ and $y$ if the subsemigroup generated by $x$ and $y$, denoted by $\langle x, y \rangle$, is not nilpotent. Now, let $S_i$, with $1 \leq i \leq n$, denote the connected ${\mathcal N}_S$-components. If $\abs{S_{i}} > 1$, for each $1 \leq i \leq n$, and if each connected ${\mathcal N}_{S}$-component is complete then each connected ${\mathcal N}_{S}$-component is a completely simple semigroup with the trivial maximal group and the semigroup $S$ is a semilattice $\Omega=\{1,\ldots,n\}$ of its ${\mathcal N}_{S}$-components (\cite[Theorem 3.7]{Jes-Sha}). In the view of Lemma~\ref{completely regular}, it follows that $S/\N=\Omega$.

The following theorem can be seen as an immediate result of the Lemma~\ref{completely regular}.

\begin{thm} \label{completely regular2}
Let $S$ be a completely regular finite semigroup. Then $S/\N$ is a Clifford semigroup.
\end{thm}


\section{An approach to $\N$ through principal series}\label{approach}

Every finite semigroup admits at least one principal series. In this section we suppose that $S$ is a finite semigroup with the principal series:
$$S= S_1 \supset S_2 \supset \ldots \supset S_{o} \supset S_{o+1} = \emptyset.$$
That is, each
$S_p$ is an ideal of $S$ and there is no ideal of $S$ strictly
between $S_p$ and $S_{p+1}$ (for convenience we call the empty set
an ideal of $S$). Each principal factor $S_p / S_{p+1} (1 \leq p \leq o)$ of $S$ is either completely $0$-simple, completely simple
or null. Every completely 0-simple factor semigroup is isomorphic
with a regular Rees matrix semigroup over a finite group $G$. 


\begin{lem} \label{diagonal T}
Suppose that $S_p / S_{p+1}$ is a non-null principal factor of $S$ and has an element $a$ such that $a \not\N\theta$.
Then there exists an integers $p \leq p' \leq o$ such that the principal factor $S_{p'} / S_{p'+1}$ satisfies the following properties: 
\begin{enumerate}
\item For every $x \in S_p \setminus S_{p+1}$ there exists an element $y \in S_{p'} \setminus S_{p'+1}$ such that $x \N y$. 
\item For every $z \in S_{p} \setminus S_{p+1} \cup S_{p'} \setminus S_{p'+1}$ and $z' \in S_{p'+1}$, $z \not\N z'$. 
\item
The principal factor $S_{p'} / S_{p'+1}$ is a $\textbf{CS}$-diagonal regular Rees matrix semigroup. 
\item If $(g;\alpha,\beta), (g';\gamma,\lambda) \in S_{p} \setminus S_{p+1}=\mathcal{M}^{0}(G,n,m;P)$ such that $p_{\beta,\gamma} \neq \theta$ and $(h;i,j), (h';k,l) \in S_{p'} \setminus S_{p'+1} = \mathcal{M}^{0}(H,n',m';Q)$ such that $(g;\alpha,\beta)  \N (h;i,j)$ and $(g';\gamma,\lambda) \N (h';k,l)$ then $k \in \mathbb{I}_{r}$ and $j \in \mathbb{I'}_{r}$ for some $ 1 \leq r\leq {n'}_{\N}$.
\end{enumerate}
\end{lem}

\begin{proof}
Since $S_p / S_{p+1}$ is not null, $S_p / S_{p+1}$ is completely simple or completely $0$-simple.
If $S_p / S_{p+1}$ is completely simple, then trivially the results is easy to verify.
Then we suppose that $S_p / S_{p+1}= \mathcal{M}^0(G,n,m;P)$
is a completely $0$-simple and $x \in S_p \setminus S_{p+1}$. We denote $\mathcal{M}^0(G,n,m;P)$ by $M$. 

Let $p'$ be the biggest integer that there exists an element $b \in S_{p'} \setminus S_{p'+1}$ such that $a \N b$. 

Since $a,x \in M$, there exist integers $1 \leq \alpha,\gamma \leq n, 1 \leq \beta,\lambda \leq m$ and elements $g,h \in G$ such that $a=(g;\alpha, \beta)$ and $x=(h;\gamma,\lambda)$.
Also since $M$ is regular, for every $1 \leq l \leq m$ and $1 \leq r \leq n$ there exist integers $1 \leq l' \leq n$ and $1 \leq r' \leq m$ such that $p_{l,l'},p_{r',r} \neq \theta$.

Let $A=(hg^{-1}p_{\alpha',\alpha}^{-1};\gamma, \alpha')$, $A'=(gh^{-1}p_{\gamma',\gamma}^{-1};\alpha, \gamma')$, $B=(p_{\beta,\beta'}^{-1};\beta', \lambda)$ and $B'=(p_{\lambda,\lambda'}^{-1};\lambda', \beta)$. 
Since $S_{p'}$ is an ideal, $AbB \in S_{p'}$. If $AbB \in S_{p'+1}$, then $A'AbBB' \in S_{p'+1}$. Since $a = A'AaBB'$, $a \N A'AbBB'$,
a contradiction with the assumption. Therefore $AbB \in S_{p'} \setminus S_{p'+1}$. Since $(h;\gamma,\lambda)=AaB$, $x \N AbB$ 
and as $x$ is taken as an arbitrary element of $S_p \setminus S_{p+1}$, for every element of $c \in S_p \setminus S_{p+1}$ there exists an element of $d \in S_{p'} \setminus S_{p'+1}$ such that $c \N d$. In a similar way as above ${p'}$ is the biggest integer that there exists an element $e \in S_{p'} \setminus S_{p'+1}$ such that $c \N e$.


Suppose that $S_{p'} / S_{p'+1}$ is null. Let $(1;\alpha, \alpha') \in M$. By above, there exists an element $v \in S_{p'} \setminus S_{p'+1}$ such that $(1;\alpha, \alpha') \N v$. Since $S_{p'} / S_{p'+1}$ is null, $vv \in S_{p'+1}$ and thus $(1;\alpha, \alpha')(1;\alpha, \alpha')$ in the $\N$-class of an element of $S_{p'+1}$, a contradiction. Therefore $S_{p'}/ S_{p'+1}$ is not null.


If there exist elements $z \in S_{p'} \setminus S_{p'+1}$ and $z' \in S_{p'+1}$ such that $z \N z'$, then, in a similar way as above, $a$ in the $\N$-class of an element of $S_{p'+1}$, a contradiction with the assumption. Then for every $z \in S_{p'} \setminus S_{p'+1}$ and $z' \in S_{p'+1}$, $z \not\N z'$.


Now suppose that $S_{p'} / S_{p'+1}$ is completely $0$-simple.
Then there exists a completely $0$-simple semigroup $\mathcal{M}^0(H,n',m';Q)$ such that  $S_{p'} / S_{p'+1}$ is isomorphic with it. We denote $\mathcal{M}^0(H,n',m';Q)$ by $M'$.

If $q_{j,i}$ and $q_{j',i}$ are non zero and $j \neq j'$ for $1 \leq i \leq n'$ and $1 \leq j,j' \leq m'$, then
$$(q_{j,i}^{-1};i,j) = \lambda_2((q_{j,i}^{-1};i,j),(q_{j',i}^{-1};i,j'),1,1),$$
$$(q_{j',i}^{-1};i,j') = \rho_2((q_{j,i}^{-1};i,j),(q_{j',i}^{-1};i,j'),1,1)$$
and thus $(q_{j,i}^{-1};i,j) \eta_2 (q_{j',i}^{-1};i,j')$. Suppose that there exists $1 \leq i' \leq n'$ such that $q_{j',i'}=\theta$ and $q_{j,i'}\neq \theta$. Since the relation $\N$ is congruence, 
$$(q_{j,i}^{-1};i,j)(q_{j,i'}^{-1}q_{j,i};i',j) \eta_2 (q_{j',i}^{-1};i,j')(q_{j,i'}^{-1}q_{j,i};i',j)$$ and thus $(1;i,j) \N \theta$. Therefore $(1;i,j)$ is in the $\N$-class an element of $S_{p'+1}$, a contradiction. Therefore $M'$ is $\textbf{CS}$-diagonal.


Suppose that
$(g;\alpha,\beta), (g';\gamma,\lambda) \in M$ such that $p_{\beta,\gamma} \neq \theta$ and elements $(h;i,j), (h';k,l) \in M'$ such that $(g;\alpha,\beta)  \N (h;i,j)$ and $(g';\gamma,\lambda) \N (h';k,l)$. Since $(g;\alpha,\beta) (g';\gamma,\lambda) \neq \theta$, we get that $(h;i,j) (h';k,l) \neq \theta$ and thus $k \in \mathbb{I}_{r}$ and $j \in \mathbb{I'}_{r}$ for some $ 1 \leq r \leq {n'}_{\N}$. 

If $S_{p'} / S_{p'+1}$ is completely simple, then $S_{p'} / S_{p'+1}$, trivially is $\textbf{CS}$-diagonal and the last result is easy to verify.
\end{proof}


The following result can be seen as an immediate corollary of Lemma~\ref{diagonal T}.

\begin{cor} \label{diagonal T'}
Suppose that $S_p / S_{p+1}$ is a non-null principal factor of $S$. The following properties hold.
\begin{enumerate}
\item If the principal factor $S_p / S_{p+1}$ is not $\textbf{CS}$-diagonal, then the subset $S_p \setminus S_{p+1}$ is in the class of $\theta$ of $\N$ or in the $\N$-class of some principal factor $S_q / S_{q+1}$ such that $q>p$.
\item If there exists an element $a \in S_p \setminus S_{p+1}$ such that $a \N\theta$, then the subset $S_p \setminus S_{p+1}$ is in the class of $\theta$ of $\N$.
\item If every principal factor $S_q / S_{q+1}$ of $S$ such that $q \geq p$ is null or is not $\textbf{CS}$-diagonal, then the subset $S_p \setminus S_{p+1}$ is in the class of $\theta$ of $\N$.
\end{enumerate}
\end{cor}

For every principal factor $S_p / S_{p+1}$ of a finite semigroup $S$ that it is not null, we say a principal factor $S_{p'} / S_{p'+1}$ is the $\N$-root of $S_p / S_{p+1}$ if it satisfies the properties of Lemma~\ref{diagonal T}. If there exists an element $a \in S_p \setminus S_{p+1}$ such that $a \N\theta$, then we say that $\theta$ is the $\N$-root of $S_p / S_{p+1}$. If $\theta$ is the $\N$-root of $S_p / S_{p+1}$, by Corollary~\ref{diagonal T'}, the subset $S_p \setminus S_{p+1}$ is in the class of $\theta$ of $\N$.

Suppose that $S_p / S_{p+1}$ isomorphic with a regular Rees matrix semigroup $\mathcal{M}^{0}(G,n,m;P)$ and $S_{p'} \setminus S_{p'+1}$ isomorphic with a regular Rees matrix semigroup $\mathcal{M}^{0}(H,n',m';Q)$. We define two functions 
$$\phi_{\N}: \{1,\ldots, n\} \rightarrow \{\mathbb{I}_{1}, \ldots,  \mathbb{I}_{{n'}_{\N}}\}$$ and
$$\phi'_{\N}: \{1,\ldots, m\} \rightarrow \{\mathbb{I'}_{1}, \ldots,  \mathbb{I'}_{{n'}_{\N}}\}$$
such that if $(g;\alpha,\beta)  \N (h;i,j)$ for $(g;\alpha,\beta) \in \mathcal{M}^{0}(G,n,m;P)$ and $(h;i,j) \in \mathcal{M}^{0}(H,n',m';Q)$, then
$$\phi_{\N}(\alpha)=\mathbb{I}_{r} \mbox{~and~} \phi'_{\N}(\beta)=\mathbb{I'}_{s}$$
if $i \in \mathbb{I}_{r}$ and $j \in \mathbb{I'}_{s}$ for some $ 1 \leq r,s \leq {n'}_{\N}$. By Lemma~\ref{diagonal T}.(1) and \ref{diagonal T}.(4) $\phi_{\N}$ and $\phi'_{\N}$ are function.

Jespers and the author \cite{Jes-Sha3} have defined a \C\ representation of $S$ when $S$ has a proper inverse Rees matrix semigroup ideal. In this paper, we extend this definition to the following case.

Suppose that $S$ has a proper $\textbf{CS}$-diagonal ideal $M=\mathcal{M}^{0}(G,m,n;P)$ such that $$M/\N \cong \mathcal{M}^0(G_{\N},n_{\N},n_{\N};I_{n_{\N}}).$$
We define the representation (a semigroup homomorphism) 
  $$\Gamma : S\longrightarrow \mathcal{T}_{\{1, \ldots, n_{\N}\} \cup \{\theta\}} ,$$
where $\mathcal{T}$ denotes  the full
transformation semigroup $\mathcal{T}_{\{1, \ldots, n_{\N}\} \cup
\{\theta\}}$ on the set $\{1, \ldots, n_{\N}\} \cup \{\theta\}$.
The definition is as follows, for $1\leq i\leq n_{\N}$ and $s\in S$,
   $$\Gamma(s)(i) = \left\{ \begin{array}{ll}
      i' & \mbox{if} ~s(g;\alpha,\beta)=(g';\alpha',\beta)  ~ \mbox{for some} ~ g, g' \in G, \, \alpha \in \mathbb{I}_i, \alpha'\in \mathbb{I}_{i'},\\ &1\leq \beta\leq m\\
       \theta & \mbox{otherwise}\end{array} \right.$$
and
      $$\Gamma (s)(\theta ) =\theta .$$
The representation $\Gamma$ is called the \C\ representation of $S$.

We prove that $\Gamma(s)$ is well-defined for every $s\in S$. Suppose that $s(g;\alpha,\beta)=(g';\alpha',\beta)$ for some $g, g' \in G$, $\alpha \in \mathbb{I}_i, \alpha'\in \mathbb{I}_{i'}, 1\leq \beta\leq m$ and $s(h;\gamma,\lambda)=(h';\gamma',\lambda)$ for some $h, h' \in G$, $\gamma \in \mathbb{I}_i, \gamma'\in \mathbb{I}_{i''}, 1\leq \lambda\leq m$. Let $\kappa \in \mathbb{I'}_{i'}$. 
As $(1;\alpha,\kappa)(g';\alpha',\beta)\neq \theta$, $(1;\alpha,\kappa)s\neq \theta$ and thus $(1;\alpha,\kappa)s=(k;\alpha,\kappa')$ for some $k\in G$ and $\kappa' \in \mathbb{I'}_{i}$. Now as $\gamma \in \mathbb{I}_i$, $(1;\alpha,\kappa)s(h;\gamma,\lambda)\neq \theta$ and thus $(1;\alpha,\kappa)(h';\gamma',\lambda)\neq \theta$. Therefore $\gamma'\in \mathbb{I}_{i'}$ and $i'=i''$.

We also claim that for $s\in S$ the map $\Gamma (s)$ restricted to the domain $S\setminus \Gamma(s)^{-1}(\theta )$ is injective.
Indeed, suppose $\Gamma(s)(i_1) = \Gamma(s)(i_2)= i$ for $1
\leq i_1, i_2, i \leq n_{\N}$.
Then $s(g;\alpha,\beta)=(g';\alpha',\beta)$ for some $g, g' \in G$, $\alpha \in \mathbb{I}_{i_1}, \alpha'\in \mathbb{I}_{i}, 1\leq \beta\leq m$ and $s(h;\gamma,\lambda)=(h';\gamma',\lambda)$ for some $h, h' \in G$, $\gamma \in \mathbb{I}_{i_2}, \gamma'\in \mathbb{I}_{i}, 1\leq \lambda\leq m$.
Let $\kappa \in \mathbb{I'}_{i}$. 
As $(1;\alpha,\kappa)(g';\alpha',\beta)\neq \theta$, $(1;\alpha,\kappa)s\neq \theta$ and thus $(1;\alpha,\kappa)s=(k;\alpha,\kappa')$ for some $k\in G$ and $\kappa' \in \mathbb{I'}_{i_1}$. Now as $(1;\alpha,\kappa)(h';\gamma',\lambda)\neq \theta$, $(1;\alpha,\kappa)s(h;\gamma,\lambda)\neq \theta$ and thus $(k;\alpha,\kappa')(h;\gamma,\lambda)\neq \theta$.
 Therefore $\gamma\in \mathbb{I}_{i_1}$ and thus $i_1=i_2$.

For every $s \in S$, $\Gamma (s)$ can be written as a product of orbits of the form
$(i_{1}, i_{2}, \ldots, i_{k})$  or of the form $(i_{1}, i_{2},
\ldots, i_{k}, \theta )$, where $1\leq i_{1}, \ldots, i_{k}\leq n_{\N}$.
The notation for  the latter orbit means that $\Gamma (s)(i_{j})=i_{j+1}$ for $1\leq
j\leq  k-1$, $\Gamma  (s)(i_{k})=\theta$, $\Gamma  (s)(\theta ) =\theta$ and
there does not exist $1\leq r \leq n_{\N}$ such that  $\Gamma
(s)(r)=i_{1}$. 
The convention, it is not written orbits of the form $(i, \theta)$ in
the decomposition of  $\Gamma (s)$  if $\Gamma(s)(i) =\theta$ and
$\Gamma(s)(j) \neq i$ for every $1\leq j \leq n_{\N}$ (this is the reason for writing orbits of length one).  If $\Gamma(s)(i)
=\theta$ for every $1\leq i \leq n_{\N}$, then it  simply is denoted $\Gamma
(s)$ as $\theta$.

If $\Gamma(s)(i_1)=i_1'$ and $\Gamma(s)(i_2)=i_2'$, then it will be shown by  $$[\ldots, i_1, i_1', \ldots,
i_2, i_2', \ldots] \sqsubseteq \Gamma(s).$$


\begin{thm} \label{diagonal T2}
Let $p$ be an integer less or equal to $o$ such that the principal factor $S_{p} / S_{p+1}$ is $\textbf{CS}$-diagonal, $S_{p} / S_{p+1}\cong\mathcal{M}^0(G,n,m;P)$ and  $(S_{p} / S_{p+1})/\N \cong \mathcal{M}^0(G_{\N},n_{\N},n_{\N};I_{n_{\N}})$.
If there exist integers $k_1, k_2,k_3$ and $k_4$ between $1$ and $n_{\N}$ and $v_1, v_2 \in S \setminus S_{p+1}$ which $k_1 \neq k_3$ and $$[\ldots, k_1, k_2, \ldots, k_3,k_4, \ldots]\sqsubseteq \Gamma(v_1), [\ldots, k_1, k_4, \ldots, k_3, k_2, \ldots] \sqsubseteq \Gamma(v_2)$$ where the representation $\Gamma$ is a \C\ representation of $S/S_{p+1}$, then one of the following properties hold:
\begin{enumerate}
\item The $\N$-root $S_{p'} / S_{p'+1}$ of the principal factor $S_{p} / S_{p+1}$ is such that $p'\neq p$, and
$\phi_{\N}(k_1)=\phi_{\N}(k_3)$, $\phi_{\N}(k_2)=\phi_{\N}(k_4)$. 
\item The $\N$-root of the principal factor $S_{p} / S_{p+1}$ is $\theta$.
\end{enumerate}
\end{thm}

\begin{proof}
First suppose that there exists an element $a \in S_{p} \setminus S_{p+1}$ such that $a \N \theta$. Then by Corollary~\ref{diagonal T'}.(2), $\theta$ is the $\N$-root of $S_p / S_{p+1}$.

Now suppose that there exists an element $a \in S_{p} \setminus S_{p+1}$ such that $a \not\N \theta$. Then by Lemma~\ref{diagonal T}.(1), $S_{p} / S_{p+1}$ has a $\N$-root $S_{p'} / S_{p'+1}$.
Let $x=(1;a_3,a_2)$ and $y=(1;a_1,a_4)$ such that $a_1\in \mathbb{I}_{k_1},a_3\in\mathbb{I}_{k_3},a_2\in\mathbb{I'}_{k_2}$ and $a_4\in\mathbb{I'}_{k_4}$.
Since  \begin{eqnarray*}
\Gamma(x)&=&\lambda_2(\Gamma(x),\Gamma(y),\Gamma(v_1),\Gamma(v_2)), \label{v1}\\
\Gamma(y)&=&\rho_2(\Gamma(x),\Gamma(y),\Gamma(v_1),\Gamma(v_2)) \label{v2}
\end{eqnarray*}
and since $n_{\N} \times n_{\N}$ and $G \times G$ are finite, there exist positive integers $t$ and $r$
  such that $t< r$ and
 \begin{eqnarray*}
 \lefteqn{
(\lambda_{t}(x,y,v_{1},v_2,v_1,v_2,\ldots),
 \rho_{t}(x,y,v_{1},v_2,v_1,v_2,\ldots)) } \\&=&
(\lambda_{r}(x,y,v_{1},v_2,v_1,v_2, \ldots), \rho_{r}(x,y,v_{1},v_2,v_1,v_2,\ldots)).
 \end{eqnarray*}
Put $w= \lambda_{t}(x,y,v_{1},v_2,v_1,v_2,\ldots)$, $z=
\rho_{t}(x,y,v_{1},v_2,v_1,v_2,\ldots)$ and $m=r-t$. Then $w =
\lambda_{m}(w, z, v_{1},v_2,v_1,v_2,\ldots) \neq z = \rho_{m}(w,
z, v_{1},v_2,v_1,v_2,\ldots)$ or $w =
\lambda_{m}(w, z, v_2,v_1,v_2,v_{1},\ldots) \neq z = \rho_{m}(w,
z,v_2,v_1,v_2,v_{1},\ldots)$ and thus $w \N z$. Since $x=(1;a_3,a_2)$ and $y=(1;a_1,a_4)$, there exist elements $g,g' \in G$, $b_1\in \mathbb{I}_{k_1},b_3\in\mathbb{I}_{k_3},b_2\in\mathbb{I'}_{k_2}$ and $b_4\in\mathbb{I'}_{k_4}$ such that $w=(g;b_3,b_2)$ and $z=(g';b_1,b_4)$ or $w=(g;b_3,b_4)$ and $z=(g';b_1,b_2)$.

We suppose that $w=(g;b_3,b_2)$ and $z=(g';b_1,b_4)$. Let $b'_3\in\mathbb{I'}_{k_3} $.
Since the relation $\N$ is congruence,
$$ (g^{-1};b_1,b'_3)(g;b_3,b_2) \N (g^{-1};b_1,b'_3)(g';b_1,b_4).$$
As $(g^{-1};b_1,b'_3)(g';b_1,b_4)\not \in S_{p} / S_{p+1}$ and $(p_{b'_3,b_3};b_1,b_2) \N (g^{-1};b_1,b'_3)(g';b_1,b_4)$, $p\neq p'$.

Now by Lemma~\ref{diagonal T}.(1), there exist elements $\alpha_1, \alpha_2 \in S_{p'} \setminus S_{p'+1}$ such that $(g^{-1};b_1,b'_3)\N\alpha_1$ and $(g';b_1,b_4)\N\alpha_2$. 
Since $$(p_{b'_3,b_3};b_1,b_2) = (g^{-1};b_1,b'_3)(g;b_3,b_2) \neq \theta$$ and $(p_{b'_3,b_3};b_1,b_2) \N \alpha_1\alpha_2$, $\phi_{\N}(b_1)=\phi_{\N}(b_3)$.
Similarly $\phi_{\N}(b_2)=\phi_{\N}(b_4)$.

If $w=(g;b_3,b_4)$ and $z=(g';b_1,b_2)$, in the similar way, the result follows.
\end{proof}

\begin{cor} \label{diagonal T2'}
Let $S= \mathcal{M}^0(G,n,m;P) \cup T$ be a finite semigroup that is the union of the ideal $M=\mathcal{M}^0(G,n,m;P)$ and the subsemigroup $T$. If the ideal $M$ is $\textbf{CS}$-diagonal, $M/\N \cong \mathcal{M}^0(G_{\N},n_{\N},n_{\N};I_{n_{\N}})$ and there exist integers $k_1, k_2,k_3$ and $k_4$ between $1$ and $n_{\N}$ and $v_1, v_2 \in T$ which  $k_1 \neq k_3$ and 
$$[\ldots, k_1, k_2, \ldots, k_3,k_4, \ldots]\sqsubseteq \Gamma(v_1), [\ldots, k_1, k_4, \ldots, k_3, k_2, \ldots] \sqsubseteq \Gamma(v_2),$$  then $\theta$ is the $\N$-root of the subsemigroup $M$.
\end{cor}

\begin{proof}
This follows at once from Theorem~\ref{diagonal T2}.
\end{proof}


Note that if a principal factor $S_{p} / S_{p+1}$ of a finite semigroup $S$ is $\textbf{CS}$-diagonal, $(S_{p} / S_{p+1})/\N \cong \mathcal{M}^0(G_{\N},n_{\N},n_{\N};I_{n_{\N}})$ and there exist integers $k_1, k_2,k_3$ and $k_4$ between $1$ and $n_{\N}$
and $v_1, v_2 \in S \setminus S_{p+1}$ such that $$[\ldots, k_1, k_2, \ldots, k_3,k_4, \ldots]\sqsubseteq \Gamma(v_1), [\ldots, k_1, k_4, \ldots, k_3, k_2, \ldots] \sqsubseteq \Gamma(v_2),$$ then in general the subset $S_{p} \setminus S_{p+1}$ is not in the class of $\theta$ of $\N$.

For example, let $S$ be a finite semigroup with principal series:
$$S= S_1 \supset S_2 \supset S_3 \supset S_4 \supset S_5= \emptyset$$
\psscalebox{.7 .7} 
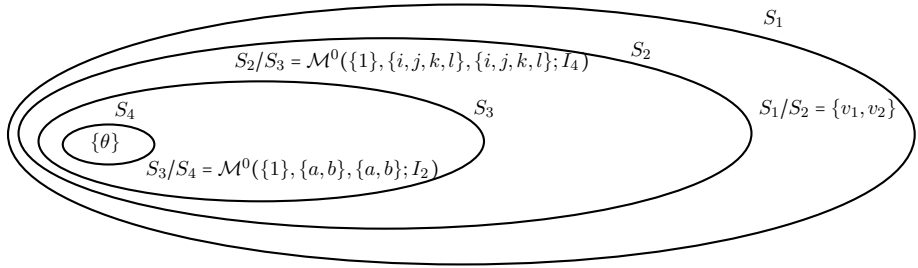
{
\begin{pspicture}(0,-2.49)(17.26,2.49)
\psellipse[linecolor=black, linewidth=0.04, dimen=outer](8.63,0.0)(8.63,2.49)
\psellipse[linecolor=black, linewidth=0.04, dimen=outer](7.18,0.02)(6.98,1.83)
\psellipse[linecolor=black, linewidth=0.04, dimen=outer](4.83,-0.13)(4.25,1.16)
\psellipse[linecolor=black, linewidth=0.04, dimen=outer](1.93,-0.19)(0.89,0.4)
\rput[bl](14.36,2.05){$S_1$}
\rput[bl](11.84,1.43){$S_2$}
\rput[bl](8.86,0.33){$S_3$}
\rput[bl](2.04,0.29){$S_4$}
\rput[bl](1.58,-0.35){$\{\theta\}$}
\rput[bl](2.64,-0.87){$S_3/S_4= \mathcal{M}^0(\{1\},\{a,b\},\{a,b\};I_2)$}
\rput[bl](4.32,1.19){$S_2 / S_3= \mathcal{M}^0(\{1\},\{i,j,k,l\},\{i,j,k,l\};I_4)$}
\rput[bl](14.28,0.29){$S_1 / S_2= \{v_1,v_2\}$}
\end{pspicture}
}

\noindent whose principal factors are as follows:
$$S_4= \{\theta\},S_3 / S_4= \mathcal{M}^0(\{1\},\{a,b\},\{a,b\};I_2),$$ $$S_2 / S_3= \mathcal{M}^0(\{1\},\{i,j,k,l\},\{i,j,k,l\};I_4)$$ and a null principal factor $S_1 / S_2= \{v_1,v_2\}$. As the maximal groups of $M_1=\mathcal{M}^0(\{1\},\{a,b\},\{a,b\};I_2)$ and $M_2=\mathcal{M}^0(\{1\},\{i,j,k,l\},\{i,j,k,l\};I_4)$ are trivial, we write the elements of $M_1$ as $(\alpha,\beta)$ for $\alpha,\beta \in \{a,b\}$ and those of $M_2$ as $[\alpha,\beta]$ for $\alpha,\beta \in \{i,j,k,l\}$. Let the function $\phi:\{i,j,k,l\}\rightarrow\{a,b\}$ be given by $\phi(i)=\phi(j)=a,\phi(k)=\phi(l)=b$. We impose the following relations on $S$: 
\begin{enumerate}
    \item[$\bullet$] for every $\alpha,\beta \in \{a,b\}$ and $\gamma,\lambda \in \{i,j,k,l\}$,
$$(\alpha,\beta)[\gamma,\lambda]=(\alpha,\beta)(\phi(\gamma),\phi(\lambda)),$$
$$[\gamma,\lambda](\alpha,\beta)=(\phi(\gamma),\phi(\lambda))(\alpha,\beta),$$
  \item[$\bullet$] for every $\alpha,\beta, \gamma,\lambda \in \{i,j,k,l\}$ if $\beta \neq \gamma$,
$$[\alpha,\beta][\gamma,\lambda]=(\phi(\alpha),\phi(\beta))(\phi(\gamma),\phi(\lambda)),$$
 \item[$\bullet$] for every $\beta \in \{a,b\}$,
$$v_1^2=v_1v_2=v_2v_1=v_2^2=\theta,$$
$$v_1(a,\beta)=v_2(a,\beta)=(b, \beta),v_1(b,\beta)=v_2(b,\beta)=\theta,$$
$$(\beta,a)v_1=(\beta,a)v_2=\theta,(\beta,b)v_1=(\beta,b)v_2=(\beta,a),$$
\item[$\bullet$] for every $\alpha \in \{i,j,k,l\}$, $\beta \in \{k,l\}$ and $\gamma \in \{i,j\},$
$$v_1[i,\alpha]=[l,\alpha],v_1[j,\alpha]=[k,\alpha],v_1[\beta,\alpha]=\theta,$$
$$v_2[i,\alpha]=[k,\alpha],v_2[j,\alpha]=[l,\alpha],v_2[\beta,\alpha]=\theta,$$
$$[\alpha,k]v_1=[\alpha,j],[\alpha,l]v_1=[\alpha,i],[\alpha,\gamma]v_1=\theta,$$
$$[\alpha,k]v_2=[\alpha,i],[\alpha,l]v_2=[\alpha,j],[\alpha,\gamma]v_2=\theta.\footnote{In order to check the associativity law for the constructed example, we used software developed in C++.}$$
\end{enumerate}

Since $\Gamma'(v_1)=(i,l,\theta)(j,k,\theta)$ and $\Gamma'(v_2)=(i,k,\theta)(j,l,\theta)$, $[i,k]\N[j,l]$ where the representation $\Gamma'$ is a \C\ representation of $S/S_3$. Then $[i,i][i,k]\N[i,i][j,l]=(a,b)$ and thus $[i,k]\N(a,b)$. Similarly, it is easy to verify that the $\N$-root of principal factor $S_2 / S_3$ is the principal factor $S_3 / S_4$ and $\phi_{\N}(i)=\phi_{\N}(j)=a,\phi_{\N}(k)=\phi_{\N}(l)=b$.

Then by Lemma~\ref{diagonal T} follows that $\N(S_2\setminus S_3) \subseteq \N(S_3\setminus S_4)$. Since $\Gamma(v_1)=\Gamma(v_2)=(a,b, \theta)$, it is easy to verify that $\mathcal{M}^0(\{1\},\{a,b\},\{a,b\};I_2) \cup \{v_1,v_2\}$ is nilpotent, where the representation $\Gamma$ is a \C\ representation of $S$. 
The $\N$-classes of $S$ are singletons except for the elements $[\gamma,\lambda],\gamma,\lambda \in \{i,j,k,l\}$ and $\theta$ which constitutes a class.

Note that the elements $v_1$ and $v_2$ satisfy the condition of Theorem~\ref{diagonal T2} of the principal factor $S_2/S_3$, but $S_2\setminus S_3$ is not in the class of $\theta$.


{\bf Open Problem.} Does there exist a finite semigroup $S$ such that it has a $\textbf{CS}$-diagonal principal factor $S_{p} / S_{p+1}$ that $$(S_{p} / S_{p+1})/\N \cong \mathcal{M}^0(G_{\N},n_{\N},n_{\N};I_{n_{\N}})$$ with the following properties:
\begin{enumerate}
\item The principal factor $S_{p} / S_{p+1}$ has
a $\N$-root  principal factor  $S_{p'} / S_{p'+1}$ such that $p \neq p'$. 
\item There do not exist exist integers $k_1, k_2,k_3$ and $k_4$ between $1$ and $n_{\N}$ with $k_1 \neq k_3$ and $v_1, v_2 \in S \setminus S_{p+1}$ which $$[\ldots, k_1, k_2, \ldots, k_3,k_4, \ldots]\sqsubseteq \Gamma(v_1), [\ldots, k_1, k_4, \ldots, k_3, k_2, \ldots] \sqsubseteq \Gamma(v_2),$$
that the representation $\Gamma$ is a \C\ representation of $S/S_{p+1}$.
\end{enumerate}

We recall that a semigroup $S$ is semisimple, if every principal factor of $S$ is 0-simple
or simple (\cite{Cli}). 
If the answer to the above open problem is negative, then every $\N$-root principal factor $S_{p} / S_{p+1}$ of a finite semisimple semigroup $S$, is a $\textbf{CS}$-diagonal and such that $(S_{p} / S_{p+1})/\N \cong \mathcal{M}^0(G_{\N},n_{\N},n_{\N};I_{n_{\N}})$ and  there do not exist exist integers $k_1, k_2,k_3$ and $k_4$ between $1$ and $n_{\N}$ with $k_1 \neq k_3$ and $v_1, v_2 \in S \setminus S_{p+1}$ which $$[\ldots, k_1, k_2, \ldots, k_3,k_4, \ldots]\sqsubseteq \Gamma(v_1), [\ldots, k_1, k_4, \ldots, k_3, k_2, \ldots] \sqsubseteq \Gamma(v_2)$$ where the representation $\Gamma$ is a \C\ representation of $S/S_p$.

If a finite semigroup $S$ is semisimple, then the classes of $S/\N$ are $\{\theta\}$ (if $S$ has $\theta$) and the $\N$-classes of $\N$-root principal factors of $S$. Using Proposition~\ref{regular Rees matrix semigroup} can be seen as a result consequence.

Recall that a semigroup is called a block group if each element has at most one inverse (for example \cite{Rho-Ste}). For instance, an inverse semigroup is the same thing as a regular block group. The collection of all finite block groups whose subgroup are nilpotent is a pseudovariety denoted $\textbf{BG}_{nil}$. Also we recall the pseudovariety $$\textbf{BI} = \{S\in \textbf{S}\mid S\mbox{~is block group and all subgroups of $S$ are trivial}\}$$ where $\textbf{S}$ is all finite semigroups.
The following theorem can be seen as a consequence of Proposition~\ref{regular Rees matrix semigroup} and Lemma~\ref{diagonal T}.

\begin{thm} \label{inverse}
Let $S$ be a finite semigroup. The quotient $S/\N$ is in the pseudovariety $\textbf{BG}_{nil}$.
\end{thm}

Note that every finite inverse semigroup whose maximal subgroups are nilpotent may not be nilpotent \cite{Jes-Sha}. Recall the semigroup $\textbf{F}_{7}$ in \cite{Ril} which is the disjoint union of the completely $0$-simple semigroup $\mathcal{M}^0(\{1\},2,2;I_{2})$ and the cyclic group $\{1,u\}$
of order $2$:
$$\textbf{F}_{7} = \mathcal{M}^0(\{1\},2,2;I_{2}) \cup \{1,u\}.$$
The multiplication of $\textbf{F}_{7}$ is defined by extending that of the defining subsemigroups via $\Gamma(1)=(1)(2)$ and $\Gamma(u)=(1,2)$. The semigroup $\textbf{F}_{7}$ is inverse but it is not nilpotent. 

From Theorem~\ref{inverse} it follows that the pseudovariety $\bf{MN}$ is a subpseudovariety of the pseudovariety $\textbf{BG}_{nil}$. 

In fact the collection of finite semigroup with empty upper non-nilpotent graph is a pseudovariety. We denoted it by $\bf{EUNNG}$. 

\begin{thm} \label{EUNNG}
The collection $\bf{EUNNG}$ is a pseudovariety.
\end{thm}
\begin{proof}
Suppose that the finite semigroup $S$ has empty upper non-nilpotent graph.
If $S'$ is a subsemigroup of $S$ and ${\mathcal N}_{S'}$ is not empty, then there exist elements $a,b \in S'$ such that $\langle a, b \rangle$ is not nilpotent. As $S'\subseteq S$, $a,b \in S$ and thus ${\mathcal N}_{S}$ is not empty, a contradiction.

Suppose that $f:S\rightarrow S'$ is a onto homomorphism and ${\mathcal N}_{S'}$ is not empty. Then there exist elements $a,b \in S'$ such that $\langle a, b \rangle$ is not nilpotent. Therefore by Lemma~\ref{finite-nilpotent}, there exists a positive integer $m$, distinct elements $x, y\in \langle a, b \rangle$ and elements $ w_{1}, w_{2}, \ldots, w_{m}\in \langle a, b \rangle^{1}$ such that $x = \lambda_{m}(x, y, w_{1}, w_{2}, \ldots, w_{m})$, $y = \rho_{m}(x,y, w_{1}, w_{2}, \ldots, w_{m})$. Then for every integer $t>0$, we also have $$x = \lambda_{tm}(x, y, w''_{1},\ldots,w''_{tm}),$$ $$y = \rho_{tm}(x,y, w''_{1},\ldots,w''_{tm})$$ that $w''_{1},\ldots,w''_{tm}$ is repeating $w_{1}, \ldots, w_{m}$, $t$ times. Since $f$ is onto, there exists elements $a',b'\in S$ such that $f(a')=a, f(b')=b$ and thus $f(\langle a', b' \rangle)=\langle f(a'), f(b') \rangle=\langle a, b \rangle$. Then there exists elements $x', y'\in \langle a', b' \rangle$ and elements $ w'_{1}, w'_{2}, \ldots, w'_{m}\in \langle a', b' \rangle^{1}$ such that $f(x')=x, f(y')=y, f(w'_{1})=w_{1}, f(w'_{2})=w_{2}, \ldots, f(w'_{m})=w_{m}$. Then $$f(x') = \lambda_{tm}(f(x'), f(y'), f(w'_{1}), f(w'_{2}), \ldots, f(w'_{m}),\ldots),$$ $$f(y') = \rho_{tm}(f(x'),f(y'), f(w'_{1}), f(w'_{2}), \ldots, f(w'_{m}),\ldots)$$ and thus $$f(x') = f(\lambda_{tm}(x', y', w'_{1}, w'_{2}, \ldots, w'_{m},\ldots)),$$ $$f(y') = f(\rho_{tm}(x',y', w'_{1}, w'_{2}, \ldots, w'_{m},\ldots))$$ for every positive integer $t$. Since the upper non-nilpotent graph of $S$ is empty, $\langle a', b' \rangle$ is nilpotent and thus there exist a positive integer $t$ such that $$\lambda_{tm}(x', y', w'_{1}, w'_{2}, \ldots, w'_{m},\ldots)=\rho_{tm}(x',y', w'_{1}, w'_{2}, \ldots, w'_{m},\ldots).$$ Then $f(x') =f(y')$ and thus $x=y$, a contradiction.

Now suppose that the semigroup $T$ has empty upper non-nilpotent graph. If ${\mathcal N}_{S\times T}$ is not empty, then there exist elements $(a_1,a_2),(b_1,b_2) \in S\times T$ such that $\langle (a_1,a_2),(b_1,b_2) \rangle$ is not nilpotent. Therefore by Lemma~\ref{finite-nilpotent}, there exists a positive integer $m$, distinct elements $(x_1,x_2),(y_1,y_2)\in \langle (a_1,a_2),(b_1,b_2) \rangle$ and elements $(w_1,w'_1), \ldots, (w_m,w'_m)\in \langle (a_1,a_2),(b_1,b_2) \rangle^{1}$ such that $$(x_1,x_2) = \lambda_{m}((x_1,x_2), (y_1,y_2), (w_1,w'_1),  \ldots, (w_m,w'_m)),$$ $$(y_1,y_2) = \rho_{m}((x_1,x_2),(y_1,y_2), (w_1,w'_1), \ldots, (w_m,w'_m)).$$ Therefore $$x_1 = \lambda_{m}(x_1, y_1 w_1,  \ldots, w_m),x_2 = \lambda_{m}(x_2, y_2, w'_1,  \ldots, w'_m),$$ $$y_1 = \rho_{m}(x_1,y_1, w_1, \ldots, w_m) \mbox{~and~} y_2 = \rho_{m}(x_2,y_2, w'_1, \ldots, w'_m).$$ Now as $(x_1,x_2)\neq(y_1,y_2)$, $x_1\neq y_1$ or $x_2\neq y_2$, a contradiction with the upper non-nilpotent graphs of $S$ and $T$ are empty. The result follows.
\end{proof}

It is easy to verify that if the upper non-nilpotent graph of the subgroup is empty, then the subgroup is nilpotent. Combining with \cite[Lemma 2.5]{Jes-Sha}, it follows that $\bf{EUNNG}$ is a subpseudovariety of $\textbf{BG}_{nil}$ and since $\bf{MN}$ is also a subpseudovariety of  $\bf{EUNNG}$, in particular $\bf{MN}$ is a subpseudovariety of  $\textbf{BG}_{nil}$. Note that the pseudovarieties $\bf{MN}$ and $\bf{EUNNG}$ are distinct \cite{Jes-Sha}. By \cite[Theorem 2.6]{Jes-Sha}, $\bf{EUNNG} \subseteq \bf{PE}$. Also by \cite[Corollary 8]{Ril}, a finite semigroup $S$ is positively Engel if and only if all its principal factors are either null semigroups or inverse semigroups over nilpotent groups, and $S$ does not have an epimorphic image with $\textbf{F}_7$ as a subsemigroup. Then $\textbf{PE} \subset \textbf{BG}_{nil}$ and thus we have $$ \textbf{MN} \subset \textbf{EUNNG} \subseteq \textbf{PE} \subset \textbf{BG}_{nil}.$$
It is easy to verify that $  \textbf{BI} \subset \textbf{PE} \subset \textbf{BG}_{nil}$
but $\textbf{MN} \not\subset \textbf{BI}$. Note that $\textbf{BI} \not\subset \textbf{MN}$ \cite{Jes-Sha}. \\


\textbf{Acknowledgments.}
This work was supported, in part, by CMUP (UID /MAT/00144/2013), which is funded by FCT (Portugal)
with national (MCTES) and European structural funds through the programs FEDER, under the partnership agreement PT2020.
It was also supported by the FCT post-doctoral scholarship SFRH/BPD/89812/ 2012. The author thanks Prof. Jorge Almeida and Prof. Bijan Davvaz for their scientific suggestions.

\bibliographystyle{plain}
\bibliography{ref-et}
\end{document}